\documentclass{amsproc}
\usepackage{amsmath,amsthm,amsfonts,amssymb,amscd,amsbsy}
\usepackage{color}
\usepackage{comment}

\newtheorem{theorem}{Theorem}[section]
\newtheorem{lemma}[theorem]{Lemma}
\newtheorem{problem}[theorem]{Problem}

\newtheorem{proposition}[theorem]{Proposition}

\theoremstyle{definition}
\newtheorem{remark}[theorem]{Remark}

\newcommand{\nn}{\mathbb{N}}

\numberwithin{equation}{section}

\begin{document}

\title[$c_{0} \widehat{\otimes}_\pi c_{0}\widehat{\otimes}_\pi c_{0}$      is not isomorphic to a subspace of $c_{0} \widehat{\otimes}_\pi c_{0}$]{\large{Solution of the problem of Diestel on $3$-fold tensor product of $c_0$}}

\author{R. M. Causey}
\address{Miami University, Department of Mathematics, Oxford, OH 45056, USA}
\email{causeyrm@miamioh.edu}

\author{E. M. Galego}
\address{University of S\~ao Paulo, Department of Mathematics, IME, Rua do Mat\~ao 1010,  S\~ao Paulo, Brazil}
\curraddr{Department of Mathematics and Statistics,}
\email{eloi@ime.usp.br}

\author{C. Samuel}
\address{Aix Marseille Universit\'e, CNRS, Centrale Marseille, I2M, Marseille, France}
\email{christian.samuel@univ-amu.fr}



\dedicatory{Dedicated to the memory of Professor Joe Diestel.}


\begin{abstract} In the present paper we prove that the $3$-fold projective tensor product of $c_0$, $c_{0} \widehat{\otimes}_\pi c_{0}\widehat{\otimes}_\pi c_{0}$,  is not isomorphic to  a subspace of  $c_{0} \widehat{\otimes}_\pi c_{0}$. In particular, this settle the long-standing open problem  of  whether   $c_{0} \widehat{\otimes}_\pi c_{0}$   is isomorphic to    $c_{0} \widehat{\otimes}_\pi c_{0}\widehat{\otimes}_\pi c_{0}$. The origin of this problem goes back to Joe Diestel who mentioned it in a private communication to the authors of paper ``Unexpected subspaces of tensor products" published in 2006.

\end{abstract}

\maketitle


\section{Introduction}

Since  Grothendieck  established the theory of tensor products \cite{Gr0},   it has been realized that the projective tensor products of Banach spaces  $X$ and $Y$ denoted by $X \widehat{\otimes}_\pi Y$   would have a great impact on the geometry of Banach spaces, but at the same time it would be complicated. In fact, today many surprising  results are known about it, 
even if the spaces involved are of simple geometric structure.

  Stehle  showed that there exists a subspace of $c_{0} \widehat{\otimes}_\pi c_{0}$  which fails to have 
the Dunford-Pettis property \cite{S}, despite being well known that every subspace of $c_{0}$  has this property \cite{Gr}, \cite{GG}. For some more unexpected facts about the geometric structure of  projective tensor product of Banach spaces,  see  for instance  \cite{CFPV}, \cite{E} and \cite{P}.

On the other hand, due to the difficulty in working with this structure of spaces, various elementary questions  on the spaces $X\widehat{\otimes}_\pi Y$ still remain unanswered. This is the case for the following problem attributed to Aleksander Pe\l czy\'nski \cite[p.517]{CFPV}. Does $c_0\widehat{\otimes}_\pi c_0$ have the uniform approximation property (UAP)? 

 Recall that a Banach $X$ has the UAP if there is a constant K and a function $f: \mathbb N \to \mathbb N$ such that, given $E \subset X$ with dim $E=k$ there is a linear continuous operator on X, with $\|T\| \leq K,$   dim $T(X) \leq f(k)$ and $T(x)=x$ for every $x\in E$ \cite{PC}. 

However, in \cite[Corollary 1.7]{CFPV} it was proved that the quadruple projective tensor product of $c_0$,  $c_0\widehat{\otimes}_\pi c_0 \widehat{\otimes}_\pi c_0\widehat{\otimes}_\pi c_0$,  does not have the UAP. Thus,  in view of Pe\l czy\'nski's problem, this last result also raised the following problem.
\begin{problem}\label{nat} Is $c_0\widehat{\otimes}_\pi c_0$ isomorphic to $c_0 \widehat{\otimes}_\pi c_0\widehat{\otimes}_\pi c_0 \widehat{\otimes}_\pi c_0?$ 
\end{problem}
  As already noted by the authors of \cite{CFPV}, by the associativity of the projective tensor product,  a positive solution to the next problem involving  the triple projective tensor projective product of $c_0$, would imply a positive solution to Problem \ref{nat}. 
\begin{problem} \label{NAT} Is $c_0\widehat{\otimes}_\pi c_0$ isomorphic to $c_0\widehat{\otimes}_\pi c_0 \widehat{\otimes}_\pi c_0?$ 
\end{problem}
 Of course, the geometric structure of $3$-fold projective  tensor products of Banach spaces is even more complicated and so far very little is known about the geometric properties of these spaces. In particular, Problem \ref{NAT} is another long time open question which is attributed to Joe Diestel \cite[p.517]{CFPV}. 

The initial motivation for studying the theme of this paper was to look for the solution of Problem \ref{NAT}. Although we have solved Problem \ref{NAT} negatively, our main result also resolves Problem \ref{nat} negatively. In fact, in Theorem \ref{mmm}  we  establish something stronger about the family of subspaces of   $c_0\widehat{\otimes}_\pi c_0$. This result is a contribution to better understand the fruitful work on projective  tensor products  started by Grothendieck in 1953  and still with many open problems related to it, see, e.g.,  \cite[Introduction]{Gi}.

\begin{theorem}\label{mmm} $c_0\widehat{\otimes}_\pi c_0\widehat{\otimes}_\pi c_0$ is not isomorphic to a subspace of $c_0\widehat{\otimes}_\pi c_0$.
\end{theorem}
In the next section, while providing some preliminaries for proving Theorem \ref{mmm}, we will also indicate the strategy for proving it.

\

Finally, observe that Theorem \ref{mmm} suggests some new questions. We only highlight one that is closely related to the subject of this work. Let $n \in \mathbb N$,  $n \geq 2$. As usual, we denote by  $\widehat{\otimes}_\pi^n c_0$ the n-fold projective tensor product of $c_0$. 
\begin{problem} \label{fff} Suppose that $\widehat{\otimes}_\pi^m c_0$ is isomorphic to $\widehat{\otimes}_\pi^n c_0$. Is it true that $m=n$?
\end{problem} 
We don't even know how to solve Problem \ref{fff} in the simplest case, i.e. $m = 3$.

\section{Preliminaries}
Our notation is standard as may be found  in \cite{Ry}. We just remember  that if  $X$ and $Y$ are  Banach spaces and 
$\mathcal{B}(X,Y)$  is the space of 
bounded bilinear functionals on $X\times Y $, then  the projective tensor norm of 
$u=\sum_{i=1}^{n} a_{i}\otimes b_{i}\in  X\otimes 
Y$ is defined by $$\Vert u \Vert =\sup\left\{\,\left\vert \sum_{i=1}^{n} 
\varphi (a_{i}, b_{i})\right \vert\,:\,\varphi \in \mathcal{B}(X,Y),\ 
\Vert \varphi  \Vert \leq 1 \, \right\}.$$

Thus, $X\widehat{\otimes}_{\pi }Y$ is the completion of $X\otimes 
Y$ with respect to this norm  \cite{Ry}. We denote by $\Vert \ \Vert _{X\widehat{\otimes}_{\pi }Y}$ the 
projective norm on $X\widehat{\otimes}_{\pi }Y.$

The idea behind our  proof  of Theorem \ref{mmm} is to argue that $c_0\hat{\otimes}_\pi c_0$ is $2$-asympto-tically uniformly smoothable (shown by Dilworth and Kutzarova in \cite[Theorem 9]{DK}), while $c_0\hat{\otimes}_\pi c_0 \hat{\otimes}_\pi c_0$ is not.   This amounts to exhibiting normalized, weakly null trees in $c_0\hat{\otimes}_\pi c_0 \hat{\otimes}_\pi c_0$ which do not admit uniform upper $\ell_2$ estimates on their branches, for which we will use the Hilbert matrices in a manner similar to Kwapien and  Pe\l czy\'{n}ski's use of the Hilbert matrices in \cite{KP}. One consequence of the results of \cite{KP} is that  $$\left \|\sum_{i=1}^n e_i\otimes s_i\right \|_{c_0\hat{\otimes}_\pi c_0} \geq \Omega \log n,$$ where $\Omega$ is a constant $> 0,$  $(e_i)_{i=1}^\infty$ is the unit vector basis of $c_0$   and $s_i=\sum_{j=1}^i e_j$ denotes the summing basis of $c_0$.  The proof proceeded by using the Hilbert matrix $h_n$ as a member of $(c_0\hat{\otimes}_\pi c_0)^*$ to norm $\sum_{i=1}^n e_i\otimes s_i$. More precisely, the Hilbert matrix $h_n$ was used to norm a tensor whose rows are a permutation of the rows of $\sum_{i=1}^n e_i\otimes s_i$. A crucial portion of that argument is the use of an appropriate upper estimate on the operator norm of $h_n$, when viewed as an operator from $c_0$ to $\ell_1$.  We will use a similar  upper estimate on the operator norm of $h_n$ from $c_0$ to $\ell_2$ and then use the Hilbert matrices (actually, row permutations of the Hilbert matrices) to norm $\sum_{i=1}^n e_i\otimes s_i\otimes f_i^n$ and provide the lower estimate 
$$\left \|\sum_{i=1}^n e_i\otimes s_i\otimes f_i^n\right \|_{c_0\hat{\otimes}_\pi c_0 \hat{\otimes}_\pi c_0} \geq  \Omega\;  n^{1/2}\log(n)$$  for a Rademacher system $(f_i^n)_{i=1}^n$ and a constant $\Omega$. Note that for each $n$, we are using a different Rademacher system $(f^n_i)_{i=1}^n$.   We then use this estimate to prove that $c_0\hat{\otimes}_\pi c_0 \hat{\otimes}_\pi c_0$ is not $2$-asymptotically uniformly smoothable, and is therefore not isomorphic to a subspace of  $c_0\hat{\otimes}_\pi c_0$.  However, we will deal with $2$-asymptotic uniform smoothness only implicitly, choosing to deal with weakly null trees instead.  We started defining this notion.

For $n\in\nn$, let $$\mathcal{A}_n=\{(m_i)_{i=1}^l: 1\leqslant l\leqslant n, m_1<\ldots<m_l, m_i\in\nn\}.$$  Given $t\in \{\varnothing\}\cup \bigcup_{l=1}^\infty \mathcal{A}_l$ and $m\in\nn$, we let $t<m$ denote the relation that either $t=\varnothing$ or $t=(m_1, \ldots, m_l)$ and $m_l<m$.   We let $\smallfrown$ denote concatenation, so that if $t=\{\varnothing\}\cup \mathcal{A}_{n-1}$ and $t<m\in\nn$, it follows that  $t\smallfrown(m)\in \mathcal{A}_n$.    

  Given a Banach space $X$, a family   $(u_t)_{t\in \mathcal{A}_n} $ of $X$  is said to be \emph{weakly null} if for any $t\in \{\varnothing\}\cup \mathcal{A}_{n-1}$, $(u_{t\smallfrown (m)})_{t<m}$ is a weakly null sequence in $X$.

For each $k\in\nn$, let $$E_k=\text{span}\{e_i\otimes e_j: \max\{i,j\}=k\}\subset c_0\hat{\otimes}_\pi c_0.$$  Then the sequence $(E_k)_{k=1}^\infty$ is a Schauder finite dimensional decomposition (FDD) for $c_0\hat{\otimes}_\pi c_0$. Moreover, since $(c_0\hat{\otimes}_\pi c_0)^*=\mathfrak{L}(c_0, \ell_1)=\mathfrak{K}(c_0, \ell_1)$, the space of  compact operators from $c_0$ to $\ell_1$, it follows that the sequence  $(E_k^*)_{k=1}^\infty$ given by $$E_k^*= \text{span}\{e_i^*\otimes e_j^*: \max\{i,j\}=k\}$$ is a FDD of   $(c_0\hat{\otimes}_\pi c_0)^*$. It was shown in \cite{DK} that  this  FDD $(E^*_k)_{k=1}^\infty$ satisfies a uniform $\ell_2$ lower estimate. That is, there exists $C_1$ such that for any $n\in\nn$,  any integers $0=r_0<r_1<\ldots <r_n$, and any $u_i\in \text{span}\{E^*_j: r_{i-1}<j\leqslant r_i\}$, $$C_1\Bigl\|\sum_{i=1}^n u_i\Bigr\|_{(c_0\hat{\otimes}_\pi c_0)^*}^2\geqslant \sum_{i=1}^n \|u_i\|^2_{(c_0\hat{\otimes}_\pi c_0)^*}.$$   By standard duality arguments, there exists a constant $C_2>0$ such that for any $k\in\nn$, any integers $0=r_0<r_1<\ldots <r_n$, and any $u_i\in \text{span}\{E_j: r_{i-1}<j\leqslant r_i\}$, $$\Bigl\|\sum_{i=1}^n u_i\Bigr\|_{c_0\hat{\otimes}_\pi c_0}^2 \leqslant C_2 \sum_{i=1}^n \|u_i\|_{c_0\hat{\otimes}_\pi c_0}^2.$$  Therefore for any $n\in\nn$, any $C_3>C_2$, and any weakly null family  $(u_t)_{t\in \mathcal{A}_n}$  of  $ B_{c_0\hat{\otimes}_\pi c_0}$, there exists $(m_1, \ldots, m_n)\in \mathcal{A}_n$ such that $$\Bigl\|\sum_{i=1}^n u_{(m_1, \ldots, m_i)}\Bigr\|_{c_0\hat{\otimes}_\pi c_0} \leqslant C_3 n^{1/2}.$$  We isolate this result in the following proposition.

\begin{proposition} There exists a constant $C$ such that for any $n\in\nn$ and any weakly null family $(u_t)_{t\in \mathcal{A}_n}$ of $ B_{c_0\hat{\otimes}_\pi c_0}$, there exists $(m_1, \ldots, m_n)\in \mathcal{A}_n$ such that $$\Bigl\|\sum_{i=1}^n u_{(m_1, \ldots, m_i)}\Bigr\|_{c_0\hat{\otimes}_\pi c_0} \leqslant C n^{1/2}.$$   

\label{DK}
\end{proposition}

We also note that the isolated property states in Proposition \ref{DK} is strictly weaker than $2$-asymptotic uniform smoothability. Moreover, we will ultimately show that $c_0\hat{\otimes}_\pi c_0\hat{\otimes}_\pi c_0$ lacks this weaker property (Proposition \ref{main}), which means we will prove something  stronger than the fact that $c_0\hat{\otimes}_\pi c_0\hat{\otimes}_\pi c_0$ is not $2$-asymptotically uniformly smoothable. In particular, joining Propositions \ref{DK} and \ref{main} it immediately follows that  $c_0\widehat{\otimes}_\pi c_0\widehat{\otimes}_\pi c_0$ is not isomorphic to a subspace of $c_0\widehat{\otimes}_\pi c_0$, so we will have proved Theorem \ref{mmm}.

\section{ The Hilbert matrices $h_n$ and  the Rademarcher system $(f_i^n)_{i=1}^n$}
In this section first  we define the \emph{Hilbert matrices} $h_n$ and some permuted versions thereof, which we denote by $p_n$. Throughout, our matrices will be identified with the operators they induce via matrix multiplication. We will denote the row $i$, column $j$ entry of a matrix $M$ by $M(i,j)$.   We define

\begin{displaymath}
   h_n(i,j) = \left\{
     \begin{array}{ll}
       \frac{1}{n+1-i-j} & : i,j\leqslant n \text{\ and\ } i+j\neq n+1\\
       0 & : \text{otherwise}
     \end{array}
   \right.
\end{displaymath} 

and 

\begin{displaymath}
   p_n(i,j) = \left\{
     \begin{array}{ll}
       \frac{1}{i-j} & : i,j\leqslant n\text{\ and\ }i\neq j\\
       0 & :\text{otherwise}.
     \end{array}
   \right.
\end{displaymath} 

\begin{remark}\label{rrr} Notice that for any $j\in\nn$, $p_n(i,j)=h_n(n+1-i,j)$ for $1\leqslant i\leqslant n$ and $p_n(i,j)=h_n(i,j)$ for all $i>n$. Therefore $\|h_n:\ell_2\to \ell_2\|=\|p_n:\ell_2\to \ell_2\|$ for all $n\in\nn$.    As noted in \cite[Inequality $1.7$]{KP}, and there attributed to Titchmarch \cite{T}, there exists a constant $\tau=\tau(2)$ such that for all $n\in\nn$, $\|h_n:\ell_2\to \ell_2\|\leqslant \tau$. Since the rows of $p_n$ are simply the rows of $h_n$ permuted, $\|p_n:\ell_2\to \ell_2\|\leqslant \tau$ for al $n\in\nn$.   Since the map $I_n:c_0\to \ell_2$ given by $I_n\sum_{i=1}^\infty a_ie_i=\sum_{i=1}^n a_ie_i$ has norm $n^{1/2}$, and since $p_n:c_0\to \ell_2$ is equal to the composition $p_nI_n:c_0\to \ell_2$,  we have
\begin{equation}\label{norm}
\|p_n:c_0\to \ell_2\|\leqslant \tau n^{1/2} \text{ for all } n\in\nn. 
\end{equation}
\end{remark}

Next we need to remember the definitions of the Haar and the Rademacher systems introduced by A. Pe\l czy\' nki and Singer \cite{PS} in a $2^n$-dimensional space with respect to a symmetric basis $(x_i)_{1\leq i\leq 2^n}.$ The Haar system $(y_i)_{1\leq i\leq 2^n}$ is the sequence defined by 
$$y_1=\sum_{i=1}^{2^n}x_1,\  y_{2^k+l}=\sum_{i=1}^{2^n} \beta_i^{k,l}x_i,\ (l=1,\dots,2^k\; ;\; k=0,\dots, n-1)
$$
where
$$
\beta_i^{k,l}=\begin{cases}
\phantom{-}1& \text{ for } (2l-2)2^{n-k-1}+1 \leq i \leq (2l-1)2^{n-k-1}\\
-1 & \text{ for } (2l-1)2^{n-k-1}+1 \leq i \leq 2l\; 2^{n-k-1}\\
\phantom{-} 0& \text{ for } 1\leq i\leq(2l-2)2^{n-k-1} \text{ and }  2l\; 2^{n-k-1}+1\leq i\leq 2^n
\end{cases}
$$
We shall call Rademacher system the sequence $(r_k)_{1\leq k\leq n}$ defined by
$$r_k=\sum_{l=1}^{2^{k-1}}y_{2^{k-1}+l}
$$
We denote $(f_i^n)_{1\leq i\leq n}$ the Rademacher system associated to the unit basis $\ell_\infty^{2^n}$ and 
$(g_i^n)_{1\leq i\leq n}$ the normalized Rademacher system associated to the unit basis $\ell_1^{2^n}.$ In the duality $\langle \ell_1^{2^n},\ell_\infty^{2^n}\rangle$ we have $g_i^n(f_i^n)=1.$

\begin{lemma}\label{ue}
For any scalars   $(a_i)_{1\leq i\leq n}$ we have
$$\left\Vert \sum_{i=1}^n a_ig_i^n\right\Vert\leq \left(\sum_{i=1}^n \vert a_i\vert^2\right)^\frac{1}{2}
$$\end{lemma}
\begin{proof} Let $(r_k)_{1\leq k}$ be the sequence of the usual Rademacher system. It follows from the claim   (10) of \cite{PS} and H\"{o}lder's inequality that
$$\left\Vert \sum_{i=1}^n a_ig_i^n\right\Vert=\int_{[0,1]}\left \vert \sum_{i=1}^n a_ir_i(t)\right \vert dt \leq \left(\int_{[0,1]}\left \vert \sum_{i=1}^n a_ir_i(t)\right  \vert^2dt\right)^\frac{1}{2}  .
$$
The usual Rademacher system is   orthonormal so   the right hand inequality follows. 
\end{proof}

The following lemma  will play a key role in section 4.

\begin{lemma}	\label{chump1} For every integer $n$ there exists a unique bounded linear operator $P_n: c_o\widehat{\otimes}_\pi c_0\to \ell_1$ such that $P_n(e_i\otimes e_j) =p_n(i,j)g_i^n.$ Moreover $\Vert P_n\Vert\leq\tau n^\frac{1}{2}.$ 
\end{lemma}
\begin{proof}
It is obvious that there exists a bilinear map $b_n$ from $  c_0\times c_0\to \ell_1$ such that $b_n(e_i,e_j)=p_n(i,j)g_i^n.$ We shall show that $b_n$ is bounded. Let $x=\sum_{i=1}^\infty a_i e_i $ and $y=\sum_{j=1}^\infty b_j e_j $ be two elements of $B_{c_0}.$ Then
$$b_n(x,y)=\sum_{i=1}^n\sum_{j=1}^na_ib_jp_n(i,j)g_i^n=
\sum_{i=1}^n a_i\left(\sum_{j=1}^nb_jp_n(i,j)\right)g_i^n.
$$
By Lemma \ref{ue} and   (\ref{norm}) it follows that
$$ \Vert b_n(x,y)\Vert\leq \left(\sum_{i=1}^n\left \vert a_i \sum_{j=1}^nb_jp_n(i,j)\right \vert^2\right)^\frac{1}{2}\leq 
\left(\sum_{i=1}^n\left \vert \sum_{j=1}^nb_jp_n(i,j)\right \vert^2\right)^\frac{1}{2}\leq\tau n^\frac{1}{2}.
$$

\end{proof}

\section{On the geometric structure of $c_0\hat{\otimes}_\pi c_0\hat{\otimes}_\pi c_0$} 

The objective of this  last section is to prove Proposition \ref{main}. It contains the fact that $c_0\hat{\otimes}_\pi c_0\hat{\otimes}_\pi c_0$ lacks the previously isolated property states in Proposition \ref{DK}.
\begin{lemma}\label{main2} There exists a constant $\Delta>0$ such that for any $n\in\nn$, $$\Bigl\|\sum_{i=1}^n e_i\otimes s_i\otimes f_i^n\Bigr\|_{\widehat{\otimes}_\pi^3 c_0} \geqslant \Delta n^{1/2}\log(n).$$  
\end{lemma}
\begin{proof}

We recall that the spaces $(c_0\hat{\otimes}_\pi c_0\hat{\otimes}_\pi c_0)^*$ and    $\mathcal{L}(c_0\hat{\otimes}_\pi c_0,\ell_1)$ are isometrically isomorphic so, for every integer $n,$
\begin{align*}
 \tau n^\frac{1}{2}\left \Vert \sum_{i=1}^n e_i\otimes s_i \otimes f_i^n\right \Vert_{{\widehat{\otimes}_\pi}^3c_0}&\geq   \left \vert \sum_{i=1}^n P_n(e_i\otimes s_i)(f_i^n)\right \vert=  \left \vert \sum_{i=1}^n\sum_{j=1}^i p_n(i,j)g_i^n(f_i^n)\right \vert\\
\noalign{\text{and letting $k=i-1, l=i-j,$} by an elementary computation we have}
&\geq \sum_{k=1}^{n-1}\sum_{l=1}^k \frac{1}{l}\geq \frac{1}{2}n \log n.
\end{align*} 
We conclude by letting $\Delta=2\tau.$
\end{proof}

\begin{proposition} There exists a constant $\Delta>0$ such that for all $n\in\nn$, there exists a weakly null tree $(u_t)_{t\in \mathcal{A}_n}$ of $ B_{\widehat{\otimes}_\pi^3 c_0}$ such that, for every $(m_1, \ldots, m_n)\in \mathcal{A}_n$, $$\Bigl\|\sum_{i=1}^n u_{(m_1, \ldots, m_i)}\Bigr\|_{\widehat{\otimes}_\pi^3 c_0} \geqslant \Delta n^{1/2}\log (n).$$

\label{main}
\end{proposition}

\begin{proof} Let $\Delta$ be the constant from Lemma \ref{main2}. Fix $n\in\nn$.  For $(m_1, \ldots, m_i)\in \mathcal{A}_n$, let $u_t=e_{m_i}\otimes s_i \otimes f_i^n\in S_{\widehat{\otimes}_\pi^3 c_0}$.  By $1$-symmetry of the unit basis of $c_0,$ for any $(m_1, \ldots, m_n)\in \mathcal{A}_n$, \begin{align*} \Bigl\|\sum_{i=1}^n u_{(m_1, \ldots, m_i)}\Bigr\|_{\widehat{\otimes}_\pi^3 c_0} & =\Bigl\|\sum_{i=1}^n e_{m_i}\otimes s_i\otimes f_i^n\Bigr\|_{\widehat{\otimes}_\pi^3 c_0}  = \Bigl\|\sum_{i=1}^n e_i\otimes s_i\otimes f_i^n\Bigr\|_{\widehat{\otimes}_\pi^3 c_0} \\ & \geqslant \Delta n^{1/2}\log(n). \end{align*}
It remains to show that $(u_t)_{t\in \mathcal{A}_m}$ is weakly null. For this, fix $t\in \{\varnothing\}\cup \mathcal{A}_{n-1}$ and let $0\leq l$ be the length of $t$. Then for each $t<m$, $u_{t\smallfrown (m)}= e_m\otimes s_{l+1}\otimes f_{l+1}^n$. Since $l$ does not depend on $m$, and since $\|s_{l+1}\|_{c_0}=\|f^n_{l+1}\|_{c_0}=1$, it follows that
 $$(u_{t\smallfrown (m)})_{t<m}=(e_m\otimes s_{l+1}\otimes f^n_{l+1})_{t<m}^\infty$$ is isometrically equivalent to the canonical $c_0$ basis in $\widehat{\otimes}_\pi^3 c_0$, and therefore a weakly null sequence. 
\end{proof}

\end{document}